\documentclass[a4paper,12pt]{amsart}
\usepackage{vmargin}
\usepackage[pdftex]{graphicx}
\usepackage{amsfonts}
\usepackage{amssymb}
\usepackage{mathrsfs}
\usepackage{stmaryrd}
\usepackage{amsmath}
\usepackage{amsthm}
\usepackage{enumerate}
\usepackage{nomencl}
\usepackage[bookmarks=true]{hyperref}   
\usepackage{esint}
\usepackage{dsfont}
\usepackage{upgreek}
\usepackage{color}

\makeatletter
\renewcommand\section{\@startsection{section}{1}{\z@}%
                       {-3\p@ \@plus -4\p@ \@minus -4\p@}%
                       {3\p@ \@plus 4\p@ \@minus 4\p@}%
                      {\normalfont\normalsize\centering\scshape}}
\makeatother

\newcommand{\todo}[1]{}

\hypersetup{
    colorlinks,%
}

\author{Lashi Bandara}
\title{Density problems on vector bundles and manifolds}
\date{26 July, 2012}
\address{Lashi Bandara, Centre for Mathematics and its Applications, 
Australian National University, Canberra, ACT, 0200, Australia}
\urladdr{\href{http://maths.anu.edu.au/~bandara}{http://maths.anu.edu.au/~bandara}}
\email{\href{mailto:lashi.bandara@anu.edu.au}{lashi.bandara@anu.edu.au}}

\keywords{Density problems, 
first order operators on vector bundles,
Laplacian on vector bundles,
second order Sobolev spaces on manifolds}
\subjclass[2010]{46E35, 53C21, 58J60}

\def\colour{\colour}

\def\colour{\color}

\newtheorem{theorem}{Theorem}[section]
\newtheorem{corollary}[theorem]{Corollary}

\newtheorem{proposition}[theorem]{Proposition}

\newtheorem{remark}[theorem]{Remark}

\newcommand{\mdot}{\cdotp}

\newcommand{\bbrac}[1]{\left[#1\right]}
\newcommand{\dbrac}[1]{\left\{#1\right\}}
\newcommand{\modulus}[1]{\left|#1\right|}
\newcommand{\set}[1]{\dbrac{#1}}

\newcommand{\dom}{ {\mathcal{D}}}

\newcommand{\R}{\mathbb{R}}
\newcommand{\C}{\mathbb{C}}

\newcommand{\Na}{\ensuremath{\mathbb{N}}}

\newcommand{\script}[1]{\mathscr{#1}}

\newcommand{\intersect}{\cap}

\newcommand{\close}[1]{\overline{#1}}		

\renewcommand{\epsilon}{\varepsilon}

\renewcommand{\phi}{\varphi}

\newcommand{\graph}{\script{G}}		

\newcommand{\tensor}{\otimes}

\newcommand{\comm}[1]{\bbrac{#1}}		

\newcommand{\norm}[1]{\left\| #1 \right\|}			

\DeclareMathOperator{\tr}{tr}			

\DeclareMathOperator{\divv}{div}		

\newcommand{\cut}{\ \llcorner\ }			

\newcommand{\Ric}{{\rm Ric}}			

\DeclareMathOperator{\inj}{inj} 		

\newcommand{\Forms}[1][{}]{\mathbf{\Omega}^{#1}}		
\newcommand{\Tensors}[1][{}]{{\mathcal{T}}^{(#1)}}	
\newcommand{\Sect}{\mathbf{\Gamma}}		
\newcommand{\tanb}{{\rm T}}		
\newcommand{\cotanb}{{\rm T}^\ast}	

\DeclareFontFamily{OT1}{restrictfont}{}
\DeclareFontShape{OT1}{restrictfont}{m}{n}{<-> fmvr8x}{}

\newcommand{\adj}[1]{{#1}^\ast}			
\newcommand{\biadj}[1]{{#1}^{\ast\ast}}		
\newcommand{\extd}{{\rm d}}			
\newcommand{\intd}{{\updelta}}
\newcommand{\Dir}{{\rm D}}			

\newcommand{\inprod}[1]{\left\langle #1 \right\rangle}	
\newcommand{\grad}{\nabla}			
\newcommand{\conn}[1][{}]{{\grad_{{#1}}}}		

\newcommand{\conj}[1]{\overline{#1}}				

\newcommand{\Lp}[2][{}]{{\rm L}^{#2}_{\rm #1}}		
\newcommand{\Ck}[2][{}]{{\rm C}^{#2}_{\rm #1}}		
\newcommand{\Sob}[2][{}]{{\rm W}^{#2}_{\rm #1}}		

\newcommand{\Hil}{\script{H}}			
\newcommand{\Lap}{\Delta}			
\newcommand{\BLap}{\Lap_{\rm B}}
\newcommand{\ALap}{\Lap_{\rm A}}

\newcommand{\coLap}{\rotatebox[origin=c]{-90}{$\Lap$}}

\newcommand{\symb}{\upsigma}
\newcommand{\cp}{{\rm c}}

\newcommand{\cD}{\mathcal{D}}

\newcommand{\cV}{\mathcal{V}}
\newcommand{\cM}{\mathcal{M}}

\newcommand{\cW}{\mathcal{W}}

\newcommand{\mg}{\mathrm{g}}
\newcommand{\mh}{\mathrm{h}}

\newcommand{\RNum}[1]{\uppercase\expandafter{\romannumeral #1\relax}}

\begin{document}

\maketitle
\vspace*{-2em}
\begin{abstract}
We study some canonical 
differential operators
on vector bundles
over smooth, complete Riemannian manifolds.
Under very general assumptions, 
we show that smooth, compactly supported
sections are dense in the domains
of these operators.
Furthermore, we show that smooth, 
compactly supported functions
are dense in second order
Sobolev spaces on such manifolds
under the sole additional assumption
that the Ricci curvature is uniformly bounded
from below.
\end{abstract}
\vspace*{-0.5em}
\tableofcontents
\vspace*{-2em}

\parindent0cm
\setlength{\parskip}{\baselineskip}

\section{Introduction}

In the analysis of differential operators,
 it is often useful in calculations
to know that smooth, compactly supported functions
are dense in the domain of the operator in question.
We call this the \emph{density problem}.
In this paper, we show that
the density problem can be solved in the positive
for some canonical differential operators
over a wide class of vector bundles.

More precisely, let $\cV$
be a smooth vector bundle over
a smooth, complete Riemannian manifold $\cM$.
Suppose that $\cV$ is equipped with
a metric $\mh$ and connection
$\conn$ that are compatible.
Defining $\divv = -\adj{\conn}$
in the $\Lp{2}$ theory, 
we show that $\Ck[c]{\infty}(\cotanb \cM \tensor \cV)$
is dense in $\dom(\divv)$.
Furthermore, letting $\BLap = -\divv\close{\conn}$,
the \emph{Bochner Laplacian} on $\cV$,
we show
that $\Ck[c]{\infty}(\cV)$
is dense in $\dom(\BLap)$.
In the case that $\cV = \Forms(\cM)$, 
the exterior algebra over $\cM$, 
we consider the operator $\extd$, 
the exterior derivative,
and its adjoint $\intd = \adj{\extd}$.
In this situation, we show that $\Ck[c]{\infty}(\Forms(\cM))$
is dense in $\dom(\intd)$.

While some of the results we present
in this paper are known and 
can be accessed via alternative methods,
our result on the density problem
for second order Sobolev spaces
on manifolds is new. The Sobolev space 
$\Sob{2,2}(\cM)$ is 
defined as the closure of 
functions $u \in \Ck{\infty} \intersect \Lp{2}(\cM)$
satisfying $\modulus{\conn u}, \modulus{ \conn^2 u} \in \Lp{2}(\cM)$
with respect to the norm 
$\norm{u}_{\Sob{2,2}} = \norm{u} + \norm{\conn u} + \|\conn^2 u\|,$
and 
$\Sob[0]{2,2}(\cM)$ 
as the closure of 
$\Ck[c]{\infty}(\cM)$
under the same norm.
According to Hebey in \cite{Hebey},
the best known
conditions
yielding $\Sob[0]{2,2}(\cM) = \Sob{2,2}(\cM)$
is to require both $\Ric \geq \eta \mg$ and
$\inj(\cM) \geq \kappa$,
for some $\eta \in \R$ and $\kappa > 0$.
We dispense
the latter bound, 
yielding the following highlight
theorem of this paper.

\begin{theorem}
\label{Thm:High} 
Let $\cM$ be a smooth, complete Riemannian manifold with
metric $\mg$ and Levi-Cevita connection $\conn$.
If there exists $\eta \in \R$ such that $\Ric \geq \eta \mg$, then 
$\Sob[0]{2,2}(\cM) = \Sob{2,2}(\cM)$. 
\end{theorem}

Our motivation to study density problems
emerges from the study of
Kato square root type problems
in the presence of geometry.
The classical version of this
problem on $\R^n$
was solved by Auscher, Hofmann, Lacy, 
McIntosh and Tchamitchian in \cite{AHLMcT}
and was rephrased in a first order
point of view in  \cite{AKMc}
by Axelsson (Ros\'en),
Keith and McIntosh.
The proofs in the latter paper particularly
exploit the density of compactly supported smooth
functions and vector fields
in the domains
of the gradient and divergence
operators respectively.
Morris in \cite{Morris3}
combines ideas from \cite{AKMc} and \cite{AKM2},
rephrases and solves similar questions
in the context of submanifolds in $\R^n$.
There, density facts were needed
but these were handled by different techniques.
In formulating and solving a 
Kato square root problem on vector bundles
by the author and McIntosh in \cite{BMc}, 
density problems became 
of central importance. While some of these issues
were circumvented by alternative means,
the desire to address density concerns 
persisted.  

The main theme and philosophy in this paper
is the following. Given a first 
order differential operator $D: \Ck{\infty}(\cV) \to
\Ck{\infty}(\cW)$ (where $\cW$ is another vector bundle),  
we construct an operator $\Pi$ 
that is symmetric on $\Ck[c]{\infty}(\cV) \oplus \Ck[c]{\infty}(\cW)$
such that it encodes $D$ in a natural way.
We then show that the density 
results follow from, or are sometimes equivalent to,
showing that $\Pi$ is essentially self-adjoint.
The density problem for second order Sobolev spaces
on manifolds follows from
the essential self-adjointness of a particular $\Pi^2$
coupled with the lower bound on Ricci curvature. 
The paper \cite{Ch} by Chernoff, 
introduced to the author by Baskin, 
justifies this reduction of density problems
to essential self-adjointness,
as it provides a set of very general conditions under which
symmetric operators and their powers
are essentially self-adjoint.
\section*{Acknowledgements}

Work on this paper commenced whilst
the author was visiting Steve Hofmann
at the University 
of Missouri, Columbia, Missouri
supported by an Australian-American 
Fulbright Scholarship
and completed while visiting 
Northwestern University, Evanston, Illinois,
supported by this institution. 
The author wishes to acknowledge
these institutions as well as
his home institution, the Australian National University. 

The author would like to acknowledge
and emphasise 
the invaluable contribution made by Dean Baskin who introduced
the author to work of Paul Chernoff
that made this paper possible.
Furthermore, the author would like
to thank Andrew Morris and Michael Munn
for the many fruitful conversations
regarding density problems.
The author would also like to acknowledge
his PhD supervisor 
Alan McIntosh
for his helpful comments, 
providing corrections,
and suggesting improvements to
some of the results of this paper.
\section{Preliminaries}

\subsection{Notation}
Throughout this paper, we assume
Einstein summation notation. Explicitly,
whenever there is a raised
and lowered index appearing multiplicatively,
we assume summation over that index.
For two quantities $a, b \geq 0$,
we express inequalities up to a
constant by writing $a \lesssim b$. By this
we mean that there exists a $C > 0$ such that  
$a \leq C b$. The constant $C$ will 
be independent of the quantities $a$ and $b$, 
and the dependence will be clear from context
or preceding hypotheses. By writing $a \simeq b$, 
we mean that $a \lesssim b$ and $b \lesssim a$.

\subsection{Operator theory}

In this section, we provide
an exposition of ideas
from operator theory that 
we use in this paper.
While some of the ideas here are
valid for operators on general 
Banach spaces, we  
restrict ourselves to 
the theory in Hilbert spaces.
We refer the reader to the
excellent books \cite{Kato}
by Kato and \cite{Yosida} by Yosida
which provide a more complete
description of operator theory.
 
Let $\Hil_1$ and $\Hil_2$ be
a Hilbert spaces with inner products
$\inprod{\mdot,\mdot}_j: \Hil_j \times \Hil_j \to \C$.
We say that a linear 
map $T: \dom(T) \subset \Hil_1 \to \Hil_2$
is an \emph{operator} with \emph{domain} $\dom(T)$.
If $S$ is an operator such that $\dom(S) \subset \dom(T)$
and $Tu = Su$ for $u \in \dom(S)$, then we 
write $S \subset T$
and say that $T$ \emph{extends} $S$.
We emphasise that an operator is characterised
by both the map and the domain.

An operator $T$ is said to be
\emph{densely-defined}
if $\close{\dom(T)} = \Hil_1$ and it is said
to be \emph{closed} if its graph,
$\graph(T) = \set{(u, Tu): u \in \dom(T)}$,
is a closed subset of $\Hil_1 \times \Hil_2$.
The latter notion is equivalent
to requiring that whenever $u_n \in \dom(T)$
such that $u_n \to u \in \Hil_1$ and 
$Tu_n \to v \in \Hil_2$, then
$u \in \dom(T)$ and $Tu = v$.
Define the \emph{operator norm}
of $T$ by  $\norm{u}_T = \norm{u}_{\Hil_1} + \norm{Tu}_{\Hil_2}$
whenever $u \in\dom(T)$.
Then, an operator $T$ is closed
if and only if $(\dom(T), \norm{\mdot}_T)$
is a Banach space.
When we say ``$X$ is dense in $\dom(T)$,''
we mean that $X \subset \dom(T)$
is dense in the operator norm of $T$.
We make a motivational remark that
the notions closed and densely-defined
are particularly useful when studying differential 
operators.

By the closed graph theorem (see Theorem 5.20 in \cite{Kato}),
a closed operator 
$T$ with $\dom(T) = \Hil$ 
is \emph{bounded}, by which we mean
there exists $C > 0$ such that $\norm{Tu} \leq C \norm{u}$
for $u \in \Hil$. Boundedness of an operator 
is equivalent to saying that it is continuous.

A notion that will be very important
in later parts is
that an operator $T$ be \emph{closable}.
By this we mean that 
$\close{\graph(T)}$
is equal to the graph of another operator $\close{T}$ called the
 \emph{closure} of
$T$. 
An operator $T$ is closable
if and only if $u_n \in \dom(T)$ with
$u_n \to 0$ and $Tu_n \to v$ implies $v = 0$. 
It is immediate that 
$T \subset \close{T}$.

We say that an operator $S:\dom(S) \subset \Hil_2 \to \Hil_1$
\emph{is adjoint} to $T$ 
if 
$
\inprod{Tu,v}_{\Hil_2} = \inprod{u,Sv}_{\Hil_1}$
for all $u \in \dom(T)$ and $v \in \dom(S)$.
For such operators, if one is
densely-defined, then the other is closable
(see Theorem 5.28 in \cite{Kato}).
This is an important
fact that we often use.
For a given $T$, there
are many operators $S$ that
are adjoint to $T$. However,
if $T$ is densely-defined, 
then there exists a unique
maximal operator $\adj{T}$
adjoint to $T$ called \emph{the adjoint} 
of $T$. By maximal, 
we mean that if $S$ is any operator
adjoint to $T$, then
$S \subset \adj{T}$.
We construct $\adj{T}$
in the following way.
Let $v \in \dom(\adj{T})$ 
if there exists $f \in \Hil_1$ 
such that 
$\inprod{Tu,v}_{\Hil_2} = \inprod{u, f}_{\Hil_1}$
for all $u \in \dom(T)$
and define $\adj{T}v = f$.
The uniqueness of $\adj{T}$
follows since $\dom(T)$ is dense 
in $\Hil$.
If further $T$ is closable,
then $\adj{T}$ is densely-defined
and closed, and furthermore,
$\biadj{T} = \close{T}$
(see Theorem 5.29 in \cite{Kato}).

An important situation arises
when $\Hil = \Hil_1 = \Hil_2$.
There, we say an operator $T$ is \emph{hermitian symmetric}
(or simply symmetric) to mean $\inprod{Tu,v} = \inprod{u,Tv}$,
for all $u, v \in \dom(T)$.
The operator $T$ is said to be 
\emph{skew-symmetric} if 
$\inprod{Tu,v} = \inprod{u, -Tv}.$
A symmetric operator
is said to be self-adjoint if 
$\adj{T} = T$.
Explicitly, this means 
$Tu = \adj{T}u$
for all $u \in \dom(T) = \dom(\adj{T})$.
Note that a self-adjoint operator
is necessarily densely-defined and closed.
A densely-defined, closable
operator $T$ is said to be \emph{essentially
self-adjoint} if $\close{T}$
is self-adjoint.
Self-adjointness will be one of the primary 
tools that we use in this paper.
Although we run the risk of labouring the point, 
we present the following description of
the maximal nature of self-adjointness.

\begin{proposition}
\label{Prop:Sym}
Let $T$ be self-adjoint, and $S$ 
a symmetric extension of $T$. Then, 
$S = T$.
\end{proposition}
\begin{proof}
First, since $T = S$ 
on $\dom(T)$, it suffices to prove that $\dom(S) \subset \dom(T)$.
By the fact that $S$ is symmetric, we have that
for all $u,v \in \dom(S)$, 
$\inprod{Su,v} = \inprod{u, Sv}$.
In particular, 
this reads
$\inprod{Su,v} = \inprod{u, Tv}$
whenever $v \in \dom(T)$.
By the construction 
of the adjoint operator, we conclude
that $u \in \dom(\adj{T})$
and that $\adj{T}u = Su$.
But by the self-adjointness of $T$,
$\dom(\adj{T}) = \dom(T)$ and so we have
that $\dom(S) \subset \dom(T)$.
\end{proof}

The following proposition 
yields the same conclusion as above,
but it is more useful since 
we only need to verify that the operators
are equal on a dense subset.

\begin{proposition}
\label{Prop:Fun}
Let $T$ and $S$ 
be self-adjoint. 
If $\cD \subset \dom(T)\cup \dom(S)$, $\cD$ is dense
 in $\dom(T)$,
and $Tu = Su$ for all $u \in \cD$, then $T = S$. 
\end{proposition}
\begin{proof}
Since $S$ is self-adjoint,
in particular it is symmetric,
and thus, 
by Proposition \ref{Prop:Sym},
it suffices to show that $T \subset S$.

Let $u \in \dom(T)$. By the hypotheses, there exists a sequence
$u_j \in \cD$ such that $u_j \to u$
and $T u_j \to T u$, and further $T u_j = S u_j$. 
Thus, we have that $u_j \to u$ and $S u_j \to Tu$.
But the self-adjointness of $S$ 
in particular means that $S$ is closed
and hence $u \in \dom(S)$ and
$Su = Tu$. This shows that $T \subset S$ as required. 
\end{proof}

\subsection{Hyperbolic equations and essential self-adjointness}

The work of Chernoff in \cite{Ch}
is central to proving the results
we present in this paper.
Thus, for the convenience of the reader,
we recall some notions from this paper.
 
Our setting is the following. 
Let $\cM$ be a smooth, complete Riemannian manifold
with smooth metric $\mg$ and volume measure $d\mu$. 
Further, let $\cV$ denote a smooth, complex vector bundle
of finite rank 
over $\cM$ with 
smooth metric $\mh$. Since $\cV$ is complex,
we assume that $\mh$ is hermitian.
We denote the fibre of $\cV$ over 
$x \in \cM$ by $\cV_x$.
When we consider real bundles,
we always implicitly identify 
them under their complexification.

By $\Sect(\cV)$ denote the
$d\mu$-measurable \emph{sections}
of $\cV$. That is,
$\xi \in \Sect(\cV)$ if $\xi:\cM \to \cV$, 
$\xi(x) \in \cV_x$
and $d\mu$-measurable in $x$. 
Then, we define $\Lp{2}(\cV)$
as the space of $\xi \in \Sect(\cV)$
such that 
$$\int_{\cM} \modulus{\xi(x)}_x^2\ d\mu(x) < \infty.$$
We note that $\Lp{2}(\cV)$ is a Hilbert
space equipped with the inner product
$$\inprod{\xi,\eta} = \int_{\cM} \mh(\xi(x),\eta(x))_x\ d\mu(x).$$ 

Next, let $L$ be a first order
differential operator on $\Ck{\infty}(\cV)$,
and let $L_c = L$ on $\Ck[c]{\infty}(\cV)$.
We recall the \emph{symbol} of 
$L$ from \cite{Ch}.
Let $x \in \cM$, $v \in \cotanb_x \cM$, 
$e \in \cV_x$ and fix
$g \in \Ck{\infty}(\cM)$ and $f \in \Ck{\infty}(\cV)$
such that $\extd g(x) = v$ and 
$f(x) = e$. Then,
the symbol of $L$ at $x$ 
in the direction $v$ acting on $e$ is
defined by 
$$
\symb(x,v)e = \symb_L(x,v)e = \comm{L,gI}f(x).$$
We note that this is really the
\emph{principal} symbol of the operator
$L$ and emphasise as in \cite{Ch} that
this definition only applies
to first order operators. 

The \emph{speed
of propagation} at $x$ 
is given by
$$ \cp(x) = \cp_L(x) = \sup_{\modulus{v} = 1} \modulus{\symb(x,v)}$$
where $\modulus{\symb(x,v)} = \sup_{\modulus{e}= 1}\modulus{\symb(x,v)e}$.
Fix $x_0 \in \cM$ and $r > 0$. Then, 
the speed of propagation of $L$ inside the ball
$B(x_0,r)$ by
$$\cp(r) = \cp_L(r) = \sup_{x \in B(x_0,r)} \cp(x).$$
With this notation at hand, we
present the following
theorem which is an immediate 
consequence of 
Theorem 2.2 in \cite{Ch}.
This result is the central 
tool that we use in this paper.

\begin{theorem}
\label{Thm:Main}
Let $\cV$ be a smooth
vector bundle with smooth metric $\mh$
over a smooth, complete Riemannian manifold $\cM$.
Let $\Pi$ be a first order differential operator
on $\Ck{\infty}(\cV)$
such that $\Pi_c = \Pi$ 
with domain $\Ck[c]{\infty}(\cV)$
is symmetric. 
Furthermore, suppose that
there exists $C > 0$ such that
$\cp_{\Pi}(x) \leq C$ for each $x \in \cM$.
Then, every power of $\Pi_c$
is essentially self-adjoint.
\end{theorem}

\begin{proof}
We apply Theorem 2.2 in \cite{Ch}
which states that if $L$
is a skew-symmetric operator on $\Ck[c]{\infty}(\cV)$
and $\int_{0}^\infty \cp_L(r)^{-1} \ dr = +\infty$, 
then every power of $-\imath L$ is essentially
self-adjoint. 

First, let $L = \imath \Pi$.
Then, it is easy to see that $L_c = L$
with domain $\dom(L_c) = \Ck[c]{\infty}(\cV)$
is skew-symmetric. 
Furthermore, an easy calculation yields
that $\symb_{L}(x,v)e = \imath\symb_{\Pi}(x,v)e$.
Therefore, 
$\modulus{\symb_{L}(x,v)e} = \modulus{\symb_{\Pi}(x,v)e}$ 
which implies that 
$\cp_{L}(r) = \cp_{\Pi}(r) \leq C$
after fixing a base point $x_0 \in \cM$.
Thus,
$$ \int_0^\infty \frac{dr}{\cp_L(r)} 
= \int_0^\infty \frac{dr}{\cp_{\Pi}(r)}
\geq \int_0^\infty \frac{dr}{C} = +\infty.$$

Putting these facts together 
demonstrates that $L$ satisfies
the hypotheses of Theorem 2.2 in \cite{Ch}
and therefore,
we conclude that every
power of the operator $\Pi_c$
is essentially self-adjoint.
\end{proof}
\section{Connections, divergence and Laplacians on vector bundles}
\label{Sect:VB}

In this section, we further assume
that $\cV$ is equipped with a \emph{connection}
$\conn$. Recall that this is a map
$\conn: \Ck{\infty}(\cV) \to \Ck{\infty}(\cotanb\cM \tensor \cV)$
satisfying the \emph{Leibniz rule} 
$$\conn(fX) = \conn f \tensor X + f\conn X$$
for all $f \in \Ck{\infty}(\cM)$, 
$X \in \Ck{\infty}(\cV)$
and where $\conn f = \extd f$ is the exterior
derivative on $\cM$.
Furthermore, let us assume that $\conn$
and $\mh$ are \emph{compatible},
by which we mean 
the following \emph{product rule} 
$$X\mh(Y,Z) = \mh(\conn[X]Y, Z) + \mh(Y,\conn[X]Z)$$
holds for all $X \in \Ck{\infty}(\tanb\cM)$,
and $Y, Z \in \Ck{\infty}(\cV)$.
By letting $\conn[c] $ be the restriction of $\conn$ to $\Ck[c]{\infty}(\cV)$,
we find that 
\begin{equation*} 
\inprod{\conn[c] u, v} = \inprod{u, -\tr \conn[c] v},
\tag{$\star$}
\label{Eqn:Adj}
\end{equation*}
for all $u \in \Ck[c]{\infty}(\cV)$ 
and $v \in \Ck[c]{\infty}(\cotanb\cM \tensor \cV)$
(see Proposition 6.1
in \cite{BMc} for a proof).
This means the operators $\conn[c]$ and $-\tr \conn[c]$
are adjoint to each other
and 
by the density of $\Ck[c]{\infty}$
in $\Lp{2}$, both these
operators are densely-defined and closable.
 
Define $-\divv = \adj{\conn[c]}$. Clearly,
$\tr \conn[c] \subset \divv$, and moreover, 
$\close{\tr\conn[c]} \subset \divv$.
Furthermore,
let $\tilde{\conn} = \adj{(-\tr\conn[c])}$.
It is clear that $\close{\conn[c]} \subset \tilde{\conn}$. 
In this section we show that
under appropriate conditions, $\close{\conn[c]} = \tilde{\conn}$.
The following proposition 
illustrates the connection of this
statement to density problems.
 
\begin{proposition}
\label{Prop:VB:Eq}
Let $\cM$ be a smooth Riemannian manifold and
$\cV$ a vector bundle 
with metric $\mh$ and connection $\conn$ 
that are compatible. Then, 
the following are equivalent:
\begin{enumerate}[(i)]
\item $\close{\tr \conn[c]} = \divv$,
\item $\close{\conn[c]} = \tilde{\conn}$,
\item $\Ck[c]{\infty}(\cV)$ is dense in $\dom(\tilde{\conn})$,
\item $\Ck[c]{\infty}(\cotanb \cM \tensor \cV)$ is dense in $\dom(\divv)$.
\end{enumerate}
\end{proposition}
\begin{proof}
Suppose that $\close{\tr \conn[c]} = \divv$.
Since $\close{\conn[c]}$ is closed and densely-defined,
$\adj{(-\divv)} = \biadj{\conn[c]} = \close{\conn[c]}$. 
Also, 
$\adj{(-\divv)} = \adj{\close{\tr \conn[c]}} = \adj{(\tr \conn[c])} = \tilde{\conn}$. 
Thus (i) implies (ii).
By similar reasoning, (ii) implies (i).
Now, we note that $\close{\conn[c]} = \tilde{\conn}$
if and only if $\Ck[c]{\infty}(\cV)$ is
dense in $\dom(\tilde{\conn})$
since $\conn[c] = \tilde{\conn}$ on $\Ck[c]{\infty}(\cV)$.
Thus, (ii) and (iii) are equivalent.
By similar argument, (iv) and (i) are equivalent to 
each other. This concludes the proof.
\end{proof}

We also consider the following
self-adjoint operators
$\BLap = -\divv \close{\conn[c]}$ with domain 
$$\dom(\BLap) = \set{u \in \Lp{2}(\cV): u \in \dom(\close{\conn[c]}),\ 
\close{\conn[c]} u \in \dom(\divv)}  \subset \dom(\close{\conn[c]})$$
and $\ALap = -\close{\tr \conn[c]} \tilde{\conn}$
with domain 
$$\dom(\ALap) = \set{u \in \Lp{2}(\cV): u \in \dom(\tilde{\conn}),\ 
\tilde{\conn}u \in \dom(\close{\tr \conn[c]})} \subset \dom(\tilde{\conn}).$$

These are two \emph{Laplacians} on the vector bundle.
Furthermore, 
by (\ref{Eqn:Adj}), 
we find that the \emph{connection Laplacian} $\Lap_c = -\tr\conn[c]^2$
with domain $\Ck[c]{\infty}(\cV)$
is densely-defined and closable.
Let $\Lap_0$ denote the closure of $\Lap_c$
and denote its domain by $\dom(\Lap_0)$.
Certainly, it is easy to see that $\Lap_0 \subset \BLap$
and $\Lap_0 \subset \ALap$
and we have the
following proposition which lists
some equivalences.

\begin{proposition}
\label{Prop:VB:Eq1}
Under the same hypotheses as Proposition
\ref{Prop:VB:Eq}, the following are equivalent: 
\begin{enumerate}[(i)]
\item $\Lap_0$ is self-adjoint,
\item $\Lap_0 = \BLap = \ALap$,
\item $\Ck[c]{\infty}(\cV)$ is dense in $\dom(\BLap)$,
\item $\Ck[c]{\infty}(\cV)$ is dense in $\dom(\ALap)$.
\end{enumerate}
\end{proposition}
\begin{proof}
If $\Lap_0$ is self-adjoint, then 
since $\Lap_0 \subset \BLap$ and $\Lap_0 \subset \ALap$ where 
$\BLap$ and $\ALap$
 are in particular symmetric, 
Proposition \ref{Prop:Sym} allows us to conclude that 
$\Lap_0 = \BLap$ and $\Lap_0 = \ALap$.
Thus, (i) implies (ii).
It is easy to see that (ii) implies (iii)
since $\Lap_0 = \BLap$ on $\Ck[c]{\infty}(\cV)$.
Similarly, (ii) implies (iv).
If (iv) holds, then $\Lap_0 = \ALap$
and $\ALap$ is self adjoint
and so (iv) implies (i).
Similarly, (iii) implies (i) and this concludes
the proof.
\end{proof}

We also define the following objects that we call
\emph{co-Laplacians}.  
First, define $\coLap = -\close{\conn[c]}\divv$ with domain
$$\dom(\coLap) = \set{u \in \Lp{2}(\cotanb\cM \tensor \cV): u \in \dom(\divv),\ 
\divv u \in \dom(\close{\conn[c]})}
\subset \dom(\divv).$$
Also, let $\tilde{\coLap} = - \tilde{\conn}\close{\tr \conn[c]}$
with domain
$$\dom(\tilde{\coLap}) = \set{u \in \Lp{2}(\cotanb\cM \tensor \cV): u \in \dom(\close{\tr\conn[c]}),\ 
\close{\tr\conn[c]} u \in \dom(\tilde{\conn})}
\subset \dom(\close{\tr\conn[c]}).$$
It is easy to see that both these operators are self-adjoint.
The corresponding \emph{connection co-Laplacian}
is then given by $\coLap_c$ 
with domain $\Ck[c]{\infty}(\cotanb\cM \tensor \cV)$. 
As for the connection Laplacian, the condition  (\ref{Eqn:Adj})
implies that $\coLap_c$ is a densely-defined, closable operator.
Thus, let $\coLap_0$ denote the closure
with domain $\dom(\coLap_0)$. It is
easy to see that $\coLap_0 \subset \coLap$
and $\coLap_0 \subset \tilde{\coLap}$. 
We have the following list of
equivalences whose proof 
is similar to that of the
analogous proposition we proved
for the Laplacians.

\begin{proposition}
\label{Prop:VB:Eq2}
Under the same hypotheses as Proposition
\ref{Prop:VB:Eq},
the following are equivalent: 
\begin{enumerate}[(i)]
\item $\coLap_0$ is self-adjoint,
\item $\coLap_0 = \coLap = \tilde{\coLap}$,
\item $\Ck[c]{\infty}(\cotanb\cM \tensor \cV)$ is dense in $\dom(\coLap)$,
\item $\Ck[c]{\infty}(\cotanb\cM \tensor \cV)$ is dense in $\dom(\tilde{\coLap})$.
\end{enumerate}
\end{proposition}

We now prove the main result of this section.
We remark that as a consequence of this theorem,
each of the statements of Propositions
\ref{Prop:VB:Eq}, \ref{Prop:VB:Eq1}, 
and \ref{Prop:VB:Eq2} hold. 

\begin{theorem}
\label{Thm:VBdense}
Let $\cM$ be a smooth, complete Riemannian manifold
with metric $\mg$
and let $\cV$ be a smooth vector bundle 
over $\cM$ with smooth metric $\mh$ and connection $\conn$. 
Suppose that $\mh$ and $\conn$ are compatible.
Then, $\close{\conn[c]} = \tilde{\conn}$
and $\Lap_0$ and $\coLap_0$ are 
self-adjoint.
\end{theorem}
\begin{proof}
Let $\Hil = \Lp{2}(\cV) \oplus \Lp{2}(\cotanb \cM \tensor \cV)$ 
and define
$$ 
\Pi = \begin{pmatrix} 
	0 & -\tr \conn \\
	\conn & 0 
	\end{pmatrix}$$
with domain $\dom(\Pi) = \Ck{\infty}(\cV) \oplus \Ck{\infty}(\cotanb \cM \tensor \cV)$.
Let $\Pi_c$ denote $\Pi$ with domain $\dom(\Pi_c) = \Ck[c]{\infty}(\cV) \oplus \Ck[c]{\infty}(\cotanb \cM \tensor \cV).$
First, we show that
$\inprod{\Pi_c u, v} = \inprod{u, \Pi_c v}$,
for $u, v \in \dom(\Pi_c)$.
Let $u = (u_1, u_2), v = (v_1, v_2) \in \dom(\Pi_c)$.
Then,
$\Pi_c u = (-\tr \conn[c] u_2, \conn[c] u_1)$
and $\Pi_c v = (-\tr \conn[c] v_2, \conn[c] v_1)$.
Therefore,
\begin{align*}
\inprod{\Pi_c u, v} &= \inprod{-\tr \conn[c] u_2, v_1} + \inprod{\conn[c] u_1, v_2},\ \text{and} \\ 
\inprod{u, \Pi_c v} &= \inprod{u_1, -\tr \conn[c] v_2} + \inprod{u_2, \conn[c] v_1}.
\end{align*}
But, by (\ref{Eqn:Adj}),
$$\inprod{-\tr \conn[c] u_2, v_1} = \inprod{u_2, \conn[c] v_1}
\quad\text{and}\quad
\inprod{\conn[c] u_1, v_2} = \inprod{u_1, -\tr \conn[c] v_2}.$$
Therefore,
$\inprod{\Pi_c u, v} = \inprod{u, \Pi_c v}.$

Next, we compute the symbol of $\Pi$. 
Fix $x \in \cM$ and $v \in \cotanb_x\cM$ and
$e = (e_1,e_2) \in \cV_x \oplus \cotanb_x\cM \tensor \cV_x$.
Let $f \in \Ck{\infty}(\cV) \oplus \Ck{\infty}(\cotanb \cM \tensor\cV)$
such that $f(x) = e$, and 
$g \in \Ck{\infty}(\cM)$ such that $\conn g = v$.
Recall the following product rules
for $\conn$ and $\tr \conn$: 
$$\conn(gf) = \conn g \tensor f + g \conn f,\quad\text{and}\quad
\tr \conn (gf) = g \tr \conn v + \tr(\conj{\conn g} \tensor f)$$
where by $\conj{\conn g}$ we mean complex conjugation.
Write $f = (f_1, f_2)$. Then,
$$g  \Pi f = g  (-\tr \conn f_2, \conn f_1)$$
and
$$ \Pi (gf)
	=  (-g \tr\conn f_2 - \tr (\conj{\conn g} \tensor f_2), \conn g \tensor f_1 + g \conn f_1).$$
Thus,
\begin{multline*}
\symb(x,v)e 
	=  \Pi(gf)(x) - (g  \Pi f)(x)  \\
	=  (- \tr(\conj{\conn g(x)} \tensor f_2(x), \conn g(x) \tensor f_1(x))
	=  (-\tr(\conj{v} \tensor e_2), v \tensor e_1)
\end{multline*}
since $f(x) = (e_1, e_2)$.
Therefore,
\begin{multline*}
\modulus{\symb(x,v)e}^2 
	= \modulus{-\tr (\conj{v} \tensor e_2)}^2 + \modulus{v \tensor e_1}^2
	\leq \modulus{v}^2 \modulus{e_2}^2 + \modulus{v}^2 \modulus{e_1}^2 \\
	= \modulus{v}^2 (\modulus{e_2}^2 + \modulus{e_1}^2)
	= \modulus{v}^2 \modulus{e}^2.
\end{multline*}
From this, it is easy to see that
$\modulus{\symb(x,v)} = \modulus{v}$ 
making $\cp(x) = 1$.

These facts demonstrate that
the hypotheses of Theorem \ref{Thm:Main}
are satisfied for $\Pi$ and so we conclude that $\Pi_c$
is essentially self-adjoint.
Letting $\Pi_0$ 
denote the closure of $\Pi_c$,
note that $$
\Pi_0  = \begin{pmatrix}
	0 & -\close{\tr\conn[c]} \\
	\close{\conn[c]} & 0
	\end{pmatrix}
\quad\text{and}\quad
\adj{\Pi}_c = \begin{pmatrix}
	0 & -\divv \\
	\tilde{\conn} & 0
	\end{pmatrix}.
$$
By the self-adjointness of $\Pi_0$,
we find that 
$\adj{\Pi}_c = \adj{\close{\Pi_0}} = \adj{\Pi_0} = \Pi_0$.
Thus,
$\close{\conn[c]} = \tilde{\conn}$ and
$\close{\tr\conn[c]} = \divv$.

Also, we have that 
$$\Pi_c^2 
	=  \begin{pmatrix}
	-\tr \conn[c]^2 & 0 \\ 
	0 & - \conn[c] \tr \conn[c]
	\end{pmatrix}
	= \begin{pmatrix}
	\Lap_c & 0 \\ 
	0 & \coLap_c
	\end{pmatrix}.$$
By Theorem 2.3, $\Pi_c^2$ is essentially self-adjoint, 
which implies that $\close{\Lap_c} = \Lap_0$
and $\close{\coLap_c} = \coLap_0$ 
are also self-adjoint.
\end{proof}

\begin{remark}
When $\cV$ is $\cM \times \C$ 
(so that the sections of this bundle are measurable functions), 
the self-adjointness of $\Lap_0$
on $\Lp{2}(\cM)$ 
under the assumption that $\cM$
is complete is a well known fact.
It was stated in \cite{Ch}
as an application, and 
the author of this paper
points to \cite{Gaff} by Gaffney
and \cite{Roel} by Roelcke
as historical references.
\end{remark} 

Denote the $(p,q)$
tensors on $\cM$ by
$\Tensors[p,q]\cM$
where $p$ is the contravarient 
index and $q$ the covarient index.
We point out the following immediate
consequence of this theorem 
for the tensor Laplacian and co-Laplacian. 

\begin{corollary}
Let $\cM$ be a smooth, complete Riemannian manifold
with metric $\mg$ and connection $\conn$
that are compatible. Then, $\Ck[c]{\infty}(\Tensors[p,q]\cM)$ is
dense in $\dom(\BLap)$ and
$\Ck[c]{\infty}(\Tensors[p,q+1]\cM)$
is dense in $\dom(\coLap)$.
\end{corollary}
\begin{remark}
We emphasise that
in the corollary, we do 
not rule out the possibility that
$\conn$ has torsion.
That is, 
the relationship $\conn[X]Y - \conn[Y]X = [X,Y]$
(where $[\mdot,\mdot]$ denotes the Lie
derivative)
may fail in general for
some $X, Y \in \Ck{\infty}(\tanb\cM)$.
\end{remark}

Next, we show some consequences of Theorem \ref{Thm:VBdense}
to Sobolev spaces on vector bundles.
Fix $k \in \Na$ (although we only really deal with $k = 1,2$). 
Let $S_k$ be the set of $u \in \Ck{\infty}(\cV) \intersect \Lp{2}(\cV)$
such that $\conn^i u \in \Ck{\infty}(\cV) \intersect \Lp{2}(\Tensors[0,i] \cM \tensor \cV)$
for $1 \leq i \leq k$.
Then, the Sobolev norm is defined as
$\norm{u}_{\Sob{k,2}} = \norm{u} + \sum_{i=1}^k \norm{\conn^i u}$
for $u \in S_k$. 
We define the Sobolev spaces $\Sob{k,2}(\cV)$
as the completion of $S_k$
under this norm. The Sobolev spaces $\Sob[0]{k,2}(\cV)$
are defined as the completion of $\Ck[c]{\infty}(\cV)$
under the same norm.

Let $\conn[2]:S_1 \to 
\Ck{\infty}\intersect \Lp{2}(\cotanb\cM \tensor\cV)$ be
defined by $\conn[2] = \conn$. Then, we
have the following operator theoretic
characterisations of $\Sob{1,2}(\cV)$ and
$\Sob[0]{1,2}(\cV)$.

\begin{corollary}
\label{Cor:Sob12}
Under the same hypotheses as Theorem \ref{Thm:VBdense},
$\conn[2]$ is a densely-defined, 
closeable operator and
$\Sob{1,2}(\cV) = \dom(\close{\conn[2]})$.
Furthermore, $\Sob[0]{1,2}(\cV) = \dom(\close{\conn[c]})$ and 
$\Sob[0]{1,2}(\cV) = \Sob{1,2}(\cV)$.
\end{corollary}
\begin{proof}
Fix $u \in S_1$ and $v \in \Ck[c]{\infty}(\cotanb\cM \tensor \cV)$.
Since $x \mapsto \mh(u(x), -\tr\conn v(x))$ is compactly 
supported, an argument similar to that of the proof of Proposition 6.1 in \cite{BMc}
shows that 
$\inprod{\conn[2]u, v} = \inprod{u, -\tr\conn v}$.
This shows that $\conn[2]$ is densely-defined and closeable.
Furthermore, it shows that $\close{\conn[2]} \subset \tilde{\conn}$, 
and it is easy to see that $\close{\conn[c]} \subset \close{\conn[2]}$.
Then by Theorem \ref{Thm:VBdense}, we have that
$\close{\conn[c]} = \tilde{\conn} = \close{\conn[2]}$.

Since $\conn[2] = \conn$ on $S_1$,
it is easy to see that $\Sob{1,2}(\cV) = \dom(\close{\conn[2]})$.
We find that $\Sob[0]{1,2}(\cV) = \dom(\close{\conn[c]})$ for
a similar reason, and this concludes the proof.
\end{proof}
\begin{remark}
\begin{enumerate}[(i)]

\item In this generality, 
do not know whether $\Sob[0]{k,2}(\cV) = \Sob{k,2}(\cV)$
for $k \geq 2$. 
In fact, even for the case $\cV = \cM \times \C$
with $\conn$ being the Levi-Cevita connection,
we use assumptions on curvature
to show that $\Sob[0]{2,2}(\cM) = \Sob{2,2}(\cM)$.
We deal with this situation in \S\ref{Sect:H2}.
The case for $k > 2$ is considered 
in Proposition 3.2 in \cite{Hebey}.

\item A more general version of this result is true.
An alternative argument can be used to 
dispense the compatibility condition.  
See Proposition 2.3 in \cite{BMc}. 
\end{enumerate} 
\end{remark}

Next, define
$\Xi_k: S_k 
\to \oplus_{i=1}^k \Ck{\infty}\intersect \Lp{2}(\Tensors[0,i]\cM \tensor \cV)$
by $\Xi_k u = (\conn^i u)_{i = 1}^k$.
The following proposition characterises
Sobolev spaces in terms of this operator.

\begin{proposition}
\label{Prop:Xi}
Under the same hypothesis as Theorem \ref{Thm:VBdense},
the operator $\Xi_k$ is closeable and 
$\Sob{k,2}(\cV) = \dom(\close{\Xi_k})$.
\end{proposition}
\begin{proof}
We note the case $k = 1$ is handled
in Corollary \ref{Cor:Sob12}.

Let $u_n \in S_k$
such that $u_n \to 0$ and $\Xi_k u_n \to (v_i)_{i=1}^k.$
Then, we have that 
$v_1 = 0$ by Corollary \ref{Cor:Sob12}.
Set $w_n = \conn u_n$ so that
$w_n \to v_1 = 0$ and $\conn w_n \to v_2$.
Again, by invoking Corollary \ref{Cor:Sob12}, 
we conclude that $v_2 = 0$.
Proceeding this way, 
we find that $v_i = 0$
for all $i = 1, \dots, k$.
Thus, $\Xi_k$ is closeable.

That $\Sob{k,2}(\cV) = \dom(\close{\Xi_k})$
follows from the fact that
$\norm{u}_{\Sob{k,2}} = \norm{u}_{\Xi_k} = \norm{u} + \norm{\Xi_k u}$
for $u \in S_k$.
\end{proof}

We conclude this section with the following 
regularity result that we use in \S\ref{Sect:H2} by 
specialising $\cV$ to 
$\cM \times \C$.
 
\begin{proposition}
\label{Prop:BLap}
Let $\cM$ be a smooth, complete Riemannian manifold
and $\cV$ a smooth vector bundle with metric $\mh$
and connection $\conn$ that are compatible.
Then, $\Sob[0]{2,2} \subset \dom(\BLap)$ and 
$\Sob{2,2}(\cV) \subset \dom(\BLap)$ continuously.
\end{proposition}
\begin{proof}
First, we show that whenever $u \in \Ck{\infty} \intersect \Lp{2}(\cV)$ 
and $\tr \conn^2 u \in \Ck{\infty} \intersect \Lp{2}(\cV)$,
then $u \in \dom(\BLap)$ and $\BLap u = -\tr\conn^2 u$.
For that, fix $v\in \Ck[c]{\infty}(\cV)$ and let 
$w \in \Ck{\infty}\intersect \Lp{2}(\cotanb\cM \tensor \cV)$
with $\tr \conn w \in \Ck{\infty} \intersect \Lp{2}(\cV)$. Then,
since $x \mapsto \mh(v(x), -\tr\conn w (x))$ is compactly supported, 
we have $\inprod{-\tr\conn w, v} = \inprod{w, \conn[c]v}$.
Furthermore, by an application of the Cauchy-Schwarz inequality,
we find that $\conn u \in \Ck{\infty} \intersect \Lp{2}(\cotanb\cM \tensor \cV)$.
Thus, upon putting $w = \conn u$, we have that
$\inprod{-\tr \conn^2 u, v} = \inprod{\conn u, \conn[c]v}$.
Again, since $x \mapsto \mg(\conn u(x), \conn[c]v(x))_x$ is 
compactly supported, we have that
$\inprod{\conn u, \conn[c]v} = \inprod{u, -\tr\conn[c]^2 v}$.
In particular, for our choice of $v$, we have that $\BLap v = -\tr\conn[c]^2 v$
and since $\BLap$ is self-adjoint, $u \in \dom(\BLap)$ and $\BLap u = -\tr\conn^2 u$.

Now, we show that $\Sob{2,2}(\cV) \subset \dom(\Lap_B)$ continuously. 
So, fix $u \in \Sob{2,2}(\cV)$ and note that
as a consequence of Proposition \ref{Prop:Xi},  there
exists $u_n \in \Ck{\infty} \intersect \Lp{2}(\cV)$
such that $u_n \to u$, and $\Xi_2 u_n \to \close{\Xi_2}u$.
In particular, $(\Xi_2 u_n)$ is Cauchy and
therefore, $\norm{\conn^2 u_n - \conn^2 u_m} \to 0$.
Coupling this observation with the fact that $|\tr\conn^2 u_n| \leq |\conn^2 u_n|$,
and since $\BLap u_n = -\tr\conn^2 u_n$, 
we have that $\BLap u_n \to v$.
But $\BLap$ is closed and hence,
$u \in \dom(\BLap)$ and $v = \BLap u$.
The continuity of the embedding follows easily.

Since $\Sob[0]{2,2}(\cV) \subset \Sob{2,2}(\cV)$
is a continuous embedding, we conclude that $\Sob[0]{2,2}(\cV) \subset \dom(\BLap)$
continuously.
\end{proof}
\section{Exterior, interior derivatives and the Dirac operator}

In this section, we 
consider the density 
problem for some natural operators over
the exterior algebra of $\cM$.
To fix notation,
let the space of $p$-forms on $\cM$
be denoted by $\Forms[p](\cM)$.
Let the fibre of the bundle $\Forms[p](\cM)$
over $x \in \cM$ be 
denoted by $\Forms[p]_x(\cM)$. 
Without causing too much confusion, 
let us denote the canonical 
extension of the metric $\mg$
on $\cM$ to $\Forms[p](\cM)$
again by $\mg$.
The \emph{exterior product} of a $p$-form
$\omega$ and $q$-form
$\eta$ is denoted by 
$\omega \wedge \eta \in \Forms[p+q](\cM)$,
and the \emph{interior} or \emph{cut}
product by $\omega \cut \eta \in \Forms[q - p](\cM)$
which is defined as the dual 
to $\wedge$.
The exterior algebra is then 
denoted by $\Forms(\cM) = \oplus_{p=0}^n \Forms[p](\cM)$.
This is a graded algebra with respect to $\wedge$.
By $\Forms_x(\cM)$, we denote 
the fibre of $\Forms(\cM)$ over $x$. 

The \emph{exterior} and \emph{interior}
derivatives are then defined
by the local expressions
$$\extd_\infty \omega = dx^i \wedge \conn[\partial_i] \omega
\quad\text{and}\quad
\intd_\infty \omega = - dx^i \cut \conn[\partial_i] \omega$$
with domains
$\dom(\extd_\infty) = \dom(\intd_\infty) = \Ck{\infty}(\Forms(\cM))$. 
On $p$-forms,
$\extd_\infty: \Ck{\infty}(\Forms[p](\cM)) \to \Ck{\infty}(\Forms[p+1](\cM))$
and $\intd_\infty: \Ck{\infty}(\Forms[p](\cM)) \to \Ck{\infty}(\Forms[p-1](\cM))$.
By $\extd_c$ and $\intd_c$, denote
the operators $\extd_\infty$ and $\intd_\infty$
but with $\dom(\extd_c) = \dom(\intd_c) = \Ck[c]{\infty}(\Forms(\cM))$.
Recall also that
$$
\inprod{\extd_c\omega, \eta}
	= \inprod{\omega, \intd_c \eta}$$
for $\omega, \eta \in \Ck[c]{\infty}(\Forms(\cM))$.
Thus, $\extd_c$ and $\intd_c$ are densely-defined, 
closable operators.
We denote the closures by $\extd_0$ and $\intd_0$
respectively.
Also, let $\extd = \adj{\intd}_c$ and
$\intd = \adj{\extd}_c$. Then,
we have the following theorem.

\begin{theorem}
\label{Thm:ExtInt}
Let $\cM$ be a smooth, complete Riemannian manifold
with smooth metric $\mg$ and Levi-Cevita connection 
$\conn$.
Then, $\extd_0 = \extd$ and $\intd_0 = \intd$.
In other words, $\Ck[c]{\infty}(\Forms(\cM))$
is dense in $\dom(\extd)$ and in $\dom(\intd)$.
Furthermore, $\Ck[c]{\infty}(\Forms(\cM))$
is dense in $\dom(\extd\intd)$ and in  
$\dom(\intd\extd)$.
\end{theorem}
\begin{proof}
Define $$\Pi_\infty = \begin{pmatrix}
	0 & \intd_\infty \\
	\extd_\infty & 0
	\end{pmatrix}$$
with $\dom(\Pi_\infty) = \Ck{\infty}(\Forms(\cM)) \oplus \Ck{\infty}(\Forms(\cM))$.
Then, let $\Pi_c$ to be 
the operator $\Pi_\infty$ with 
$\dom(\Pi) =\Ck[c]{\infty}(\Forms(\cM)) \oplus \Ck[c]{\infty}(\Forms(\cM)).$

First, we show that $\Pi_c$ is symmetric.
Fix $u = (u_1, u_2), v = (v_1, v_2) \in \Ck[c]{\infty}(\Forms(\cM))$, 
and note that $\Pi_c u = (\intd_c u_2, \extd_c u_1)$,
$\Pi_c v = (\intd_c v_2, \extd_c v_1)$. Then,
$$
\inprod{\Pi_c u, v}
	= \inprod{\intd_c u_2, v_1} + \inprod{\extd_c u_1, v_2}
	= \inprod{u_2, \extd_c v_1} + \inprod{u_1, \intd_c v_1}
	= \inprod{u, \Pi_c v}.$$

Now, fix $v \in \cotanb \cM$ and $e \in \Forms_x(\cM)$.
Let $g \in \Ck{\infty}(\cM)$ and $f \in \Ck{\infty}(\Forms(\cM)) \oplus \Ck{\infty}(\Forms(\cM))$
such that $\extd g(x) = v$ and $f = (f_1, f_2)$
and $f(x) = e = (e_1, e_2)$.
Then, 
$$\symb(x,v)e 
	= \comm{\Pi_\infty, gI}f(x)
	= \Pi_\infty(gf) - g\Pi_\infty f
	= (\intd_\infty(g f_2) - g\intd_\infty f_2, \extd_\infty(g f_1) - g\extd_\infty f_1).$$
Note then that, 
$$
\extd_\infty(g f_1) - g\extd_\infty f_1
	= (\extd_\infty g \wedge f_1 + g \extd_\infty f_1) - g \extd_\infty f_1
	= \extd_\infty g \wedge f_1.$$
Thus, at $x$, $\extd_\infty g \wedge f_1 = v \wedge e_1$.
By a similar calculation,
$$
\intd_\infty(g f_2) - g \intd_\infty f_2
	= (g \intd_\infty f_2 - \extd_\infty g \cut f_2) - g \intd_c f_2
	= - \extd_\infty g \cut f_2$$
so that at $x$,
$\extd_\infty g \cut f_2 = v \cut e_2$.
We combine these two calculations 
so that
$$
\modulus{\symb(x,v)e}^2
	= \modulus{v \wedge e_1}^2 + \modulus{v \cut e_2}^2
	\leq \modulus{v}^2(\modulus{e_1}^2 + \modulus{e_2}^2)
	\leq \modulus{v}^2 \modulus{e}^2.$$

Therefore,
$\modulus{\symb(x,v)} = \modulus{v}$ and 
thus, $\cp(r) = 1$.
Therefore, by Theorem \ref{Thm:Main}, 
we conclude that $\Pi_c$
and powers of $\Pi_c$ are
essentially self-adjoint.
Thus,
$\adj{\Pi}_c = \adj{\close{\Pi}}_c
	= \close{\Pi_c}$
and since 
$$\adj{\Pi}_c = \begin{pmatrix}
		0 & \intd \\
		\extd & 0 
		\end{pmatrix},$$
we conclude that $\close{\extd}_c = \extd$ and 
$\close{\intd}_c = \intd$.

Next, note that
$$\Pi_c^2 = \begin{pmatrix} 
		\intd_c \extd_c & 0 \\
		0 & \extd_c \intd_c
		\end{pmatrix}$$
and by similar reasoning as above, 
we have that $\close{\intd_c \extd_c}$
and $\close{\extd_c\intd_c}$ are self-adjoint.
It is easy to see that $\close{\intd_c \extd_c}
\subset \intd\extd$
and $\close{\extd_c \intd_c} \subset \extd\intd$.
Furthermore, 
$\intd\extd$ and $\extd\intd$ are self-adjoint
and so by Proposition \ref{Prop:Sym}, we conclude
that
$\close{\intd_c\extd_c} = \intd\extd$
and 
$\close{\extd_c\intd_c} = \extd\intd$.
\end{proof} 

Also, define
$\Dir_\infty: \Ck{\infty}(\Forms(\cM)) \to \Ck{\infty}(\Forms(\cM))$
by $\Dir_\infty = \extd_\infty + \intd_\infty$
and $\Dir_c = \extd_c + \intd_c$
with domain $\dom(\Dir_c) = \Ck[c]{\infty}(\Forms(\cM))$.
An easy calculation shows that
$$\inprod{\Dir_c \omega, \eta} = \inprod{\omega, \Dir_c \eta}$$
when $\omega, \eta \in \Ck[c]{\infty}(\Forms(\cM))$.
Thus, $\Dir_c$ is a densely-defined,
closable operator. Denote its closure
by $\Dir_0$ and let $\Dir = \extd + \delta$. 
The latter operator is the \emph{Hodge-Dirac}
operator with domain $\dom(\Dir) = \dom(\extd) \intersect \dom(\delta)$.
As a consequence of Theorem \ref{Thm:ExtInt}, $\Dir$ is self-adjoint.
The \emph{Hodge-Laplacian} is then denoted by 
$\Dir^2$. We have the following theorem
which is well known and
written as an application in 
\cite{Ch} but  we include
it here for completeness.
 
\begin{theorem}
Let $\cM$ be a smooth, complete Riemannian manifold with smooth 
metric $\mg$ endowed with the Levi-Cevita connection
$\conn$.
Then, we have that $\Dir_0 = \Dir$
and $\close{\Dir^2_c} = \Dir^2$.
In other words, $\Ck[c]{\infty}(\Forms(\cM))$
is dense in $\dom(\Dir)$ and in $\dom(\Dir^2)$.
\end{theorem}
\begin{proof}
It is straightforward to check that
$\Dir_c$ is symmetric. 
The computation of the symbol was
done in the proof of Theorem \ref{Thm:ExtInt}.
Explicitly
$$\symb(x,v)e = v \wedge e - v \cut e,$$
and thus $c(r) = 1$.
Thus, by Theorem \ref{Thm:Main}, 
we have that $\Dir_0$ is self-adjoint. 
But $\Dir_0 \subset \Dir$ and $\Dir$
is self-adjoint so by Proposition \ref{Prop:Sym},
we have that $\Dir_0 = \Dir$. Similarly,
$\close{\Dir_c^2} = \Dir^2.$ 
\end{proof} 
\section{Second order Sobolev spaces on manifolds}
\label{Sect:H2}

Throughout this section, in addition to the smoothness and completeness 
assumptions on $\cM$ and $\mg$,
we further assume that $\cM$ is endowed
with the Levi-Cevita connection $\conn$.
In this setting, we show that $\Sob[0]{2,2}(\cM) = \Sob{2,2}(\cM)$
whenever $\Ric \geq \eta \mg$ for some $\eta \in \R$.
The best currently 
known result according to 
Hebey in \cite{Hebey} 
requires the additional
assumption that $\inj(\cM) \geq \kappa > 0$.
This is the content of Proposition 3.3 in \cite{Hebey}.

Let us recall some notation from \S\ref{Sect:VB},
but specialised to the case that $\cV$ 
is $\cM \times \C$. Let $\conn[c] = \conn$
and $\Lap = -\tr\conn^2: \Ck{\infty}(\cM) \to \Ck{\infty}(\cM)$ 
denote the Laplacian on smooth functions. 
Furthermore, let $\Lap_c = \Lap$ with $\dom(\Lap_c) = \Ck[c]{\infty}(\cM)$.
Since this is a densely-defined, closeable operator,
let $\Lap_0$ be its closure with domain 
$\dom(\Lap_0)$.
Recall that 
the Bochner Laplacian 
is the self-adjoint operator given by $\Lap_B = -\divv\close{\conn_c}$
with domain $\dom(\Lap_B)$.  

With this notation, we prove the highlight theorem of this
paper.

\begin{proof}[Proof of Theorem \ref{Thm:High}]
We note that the assumptions made in Theorem \ref{Thm:High}
satisfy the hypotheses of Theorem \ref{Thm:VBdense}
and Proposition \ref{Prop:BLap}
with $\cV$ as $\cM \times \C$.
By the latter result, we have that $\Sob[0]{2,2}(\cM) \subset \dom(\BLap)$
and $\Sob{2,2}(\cM) \subset \dom(\BLap)$ continuously.

Next, recall the following
Bochner-Lichnerowicz-Weitzenb\"ock
identity
$$
\inprod{\conn \Lap u, \conn u} = \frac{1}{2} \Lap\modulus{\conn u}^2
+ \modulus{\conn^2 u}^2 + \Ric(\conn u, \conn u)
$$
for $u \in \Ck{\infty}(\cM)$.
By using the bound $\Ric \geq \eta \mg$, 
we obtain that 
$\norm{\conn^2 u} \lesssim \norm{(I + \Lap)u}$
whenever $u \in \Ck[c]{\infty}(\cM)$.
Thus, coupled with the fact that
$\Lap_0 = \BLap$ by Theorem \ref{Thm:VBdense}, 
we find the 
the reverse inclusion $\dom(\Lap_B) \subset \Sob[0]{2,2}(\cM)$
holds and this allows us to conclude that $\dom(\Lap_B) = \Sob[0]{2,2}(\cM)$.

Combining these set inclusions, we have that 
$\Sob{2,2}(\cM) \subset \dom(\Lap_B) = \Sob[0]{2,2}(\cM).$
But $\Sob[0]{2,2}(\cM) \subset \Sob{2,2}(\cM)$
continuously and so we conclude that $\Sob[0]{2,2}(\cM) = \Sob{2,2}(\cM)$.
\end{proof}

\bibliographystyle{amsplain}
\def\cprime{$'$}
\providecommand{\bysame}{\leavevmode\hbox to3em{\hrulefill}\thinspace}
\providecommand{\MR}{\relax\ifhmode\unskip\space\fi MR }
\providecommand{\MRhref}[2]{%
  \href{http://www.ams.org/mathscinet-getitem?mr=#1}{#2}
}
\providecommand{\href}[2]{#2}


\setlength{\parskip}{0mm}

\end{document}